\documentclass[11pt,notitlepage]{article}

\usepackage{amssymb}

\usepackage[dvips]{graphicx}
\usepackage{amsmath}

\newtheorem{theorem}{Theorem}[section]

\newtheorem{conjecture}[theorem]{Conjecture}
\newtheorem{corollary}[theorem]{Corollary}

\newtheorem{example}[theorem]{Example}

\newtheorem{lemma}[theorem]{Lemma}

\newtheorem{proposition}[theorem]{Proposition}
\newtheorem{remark}[theorem]{Remark}

\newenvironment{proof}[1][Proof]{\textbf{#1.} }{\ \rule{0.5em}{0.5em}}

\newcommand{\vs}[1]{\langle #1 \rangle}
\newcommand{\Q}{{\mathbb Q}}
\newcommand{\Z}{{\mathbb Z}}

\begin{document}

\title{On the product of vector spaces in a commutative field extension}

\date{}
\author{Shalom Eliahou, Michel Kervaire and C\'{e}dric Lecouvey}
%EndAName
%LMPA Joseph Liouville, FR CNRS 2956\\
%Universit\'{e} du Littoral C\^{o}te d'Opale\\
%50 rue F. Buisson, B.P. 699\\
%F-62228 Calais cedex, France}
\maketitle

\begin{abstract}
Let $K \subset L$ be a commutative field extension. Given 
$K$-subspaces $A,B$ of $L$, we consider the subspace $\langle AB \rangle$ 
spanned by the product set $AB=\{ab \mid a \in A, b \in B\}$. 
If $\dim_K A = r$ and $\dim_K B = s$, how small can the dimension of
$\vs{AB}$ be? In this paper we give a complete answer to this
question in characteristic 0, and more generally for separable extensions.
The optimal lower bound on $\dim_K \vs{AB}$ turns out, in this case, 
to be provided by the numerical function
$$
\kappa_{K,L}(r,s) = \min_{h} (\lceil r/h\rceil + \lceil s/h\rceil -1)h,
$$
where $h$ runs over the set of $K$-dimensions of all finite-dimensional
intermediate fields $K \subset H \subset L$. This bound is closely related
to one appearing in additive number theory.
\end{abstract}

\section{Introduction}

Let $K \subset L$ be an extension of commutative fields. Let $A,B \subset L$ be non-zero $K$-subspaces of $L$. We denote by
$$
\vs{AB}
$$
the $K$-subspace of $L$ generated by the product set
\begin{equation*}
AB = \{ab \mid a \in A, b\in B \}.
\end{equation*}

Of course, if $A,B$ are finite-dimensional, then so is $\vs{AB}$ which satisfies the easy estimates
$$
\max\{\dim_K A,\dim_K B\} \;\le\; \dim_K \vs{AB} \;\le\; (\dim_K A)(\dim_K B).
$$
The above lower bound is sharp in the very special circumstance $A=B=H$ where $H$ is an intermediate field extension $K \subset H \subset L$. But in general, if $\dim_K A, \dim_K B$ are specified in advance, how small can $\dim_K \vs{AB}$ be? In other words, given positive integers $r,s \le \dim_K L$, we define
$$
\mu_{K,L}(r,s) = \min \{\dim_K \vs{AB}\},
$$
where the minimum is taken over all $K$-subspaces $A,B$ of $L$ satisfying $$\dim_K A = r, \quad \dim_K B = s.$$ For example, one has $\mu_{K,L}(h,h)=h$ whenever $h = [H:K] = \dim_K H$ is the degree of a finite-dimensional intermediate field extension $K \subset H \subset L$. 

Perhaps surprisingly, the combinatorial function $\mu_{K,L}(r,s)$ can be explicitly determined for arbitrary $r,s$ under mild hypotheses, as we do here. Our answer is provided by the following numerical function. Define 
$$
\kappa_{K,L}(r,s) = \min_{h} (\lceil r/h\rceil + \lceil s/h\rceil -1)h,
$$
where $h = [H:K]$ runs over the set of $K$-dimensions of all finite-dimensional
intermediate fields $K \subset H \subset L$. 

For example, if $[L:K]$ is a prime number $p$, then the only admissible values for $h=[H:K]$ are 1 and $p$, whence $\kappa_{K,L}(r,s) = \min\{r+s-1,p\}.$ (See Example~\ref{cauchy}.) We shall prove the following result.
\begin{theorem}\label{kappaKL} Let $K \subset L$ be a commutative field extension in which every algebraic element of $L$ is separable over $K$. Then, for all positive integers $r,s \le \dim_K L$, we have
$$
\mu_{K,L}(r,s) = \kappa_{K,L}(r,s).
$$
\end{theorem}

There are close links between this result and additive number theory, as explained in Section~\ref{link}. The proof of Theorem~\ref{kappaKL} is split between Sections 3 and 4. After some examples in Section~\ref{examples}, we look more closely, in Section~\ref{galois}, at the case of finite Galois extensions. In the last two sections, we discuss the separability hypothesis in Theorem~\ref{kappaKL}.

\section{Links with additive number theory}\label{link}

The question explored in this paper is analogous to a classical one in groups, namely that of minimizing the cardinality of product sets $AB$ where $A,B$ run over all subsets of cardinality $r,s$ in a given group $G$. In multiplicative notation, this amounts to study the function
$$
\mu_G(r,s) = \min\{|AB| \; : \;  A,B \subset G, |A| = r, |B| = s\}.
$$ 
While unknown in general, this function has recently been fully determined in the abelian case. The answer is expressed in terms of the numerical function $\kappa_G(r,s)$ defined as follows. For any group $G$, let $\mathcal{H}(G)$ be the set of orders of finite subgroups of $G$, and set
$$
\kappa_G(r,s) = \min_{h \in \mathcal{H}(G)} (\lceil r/h\rceil + \lceil s/h\rceil -1)h
$$
for all positive integers $r,s \le |G|$. Here is the result obtained in \cite{abelian}.

\begin{theorem}\label{abelian} Let $G$ be an arbitrary abelian group. Then, for all positive integers $r,s \ge 1$, we have
$
\mu_G(r,s) = \kappa_G(r,s).
$
\end{theorem}

For instance, this contains the well-known Cauchy-Davenport theorem for cyclic groups $G$ of prime order $p$, namely $\mu_G(r,s) = \min\{r+s-1, p\}$ for all $1 \le r,s \le p$. See \cite{survey} for a survey of recent results on $\mu_G(r,s)$. 

The function $\kappa_G(r,s)$ appears in various guises and contexts, for instance as the Hopf-Stiefel function $r \circ s$ in algebraic topology or in the theory of quadratic forms. See \cite{hopfstiefel} for a survey on this ubiquitous function.

\smallskip

The reader will notice the close resemblance between Theorems ~\ref{kappaKL} and ~\ref{abelian}. The methods of proof are also quite similar. In order to prove that the kappa-function is a lower bound, the key tools are a theorem of Kneser for abelian groups \cite{kneser}, and a linear version of it for separable extensions~\cite{xiang}. Regarding the optimality of the bound, the key tool is the \textit{small sumsets property}, amounting to the inequality $\mu_G(r,s) \le r+s-1$ for abelian groups \cite{abelian}. The analogous estimate for field extensions $K \subset L$, namely $\mu_{K,L}(r,s) \le r+s-1$, plays the same role and will be shown to hold in full generality.

In Section~\ref{galois}, we shall see that both versions of the kappa-function, namely $\kappa_G$ for a group $G$ and $\kappa_{K,L}$ for a field extension $K \subset L$, actually coincide for finite Galois extensions with abelian Galois group $G$.

For general background on commutative field extensions and on additive number theory, we refer to \cite{lang} and \cite{nathanson}, respectively.

\section{Proof that $\kappa_{K,L}$ is a lower bound}\label{lower}
We now go back to the field extension setting. In order to prove inequality $\mu_{K,L}(r,s) \ge \kappa_{K,L}(r,s)$ of Theorem~\ref{kappaKL}, we shall need the following linear version \cite{xiang} of a famous theorem of Kneser \cite{kneser} in additive number theory.

\begin{theorem}[Hou, Leung and Xiang]
\label{linear Kneser} \label{TH_HLX} Let $K\subset L$ be a commutative field
extension in which every algebraic element of $L$ is separable over $K$. Let $A,B\subset L$ be nonzero finite-dimensional $K$-subspaces of $L$. Let $H$ be the stabilizer of $\vs{AB}$. Then 
\begin{equation*}
\dim _{K}\langle AB\rangle \geq \dim _{K}A+\dim _{K}B-\dim _{K}H.
\end{equation*}
\end{theorem}

The separability hypothesis of the above theorem is discussed in Section~\ref{conjectures}.

\bigskip

\begin{proof}[Proof of inequality $\mu_{K,L}(r,s) \ge \kappa_{K,L}(r,s)$ of Theorem~\ref{kappaKL}] Let $A,B \subset L$ be $K$-subspaces of $L$ with $\dim_K A =r$, $\dim_K B = s$. We must prove that $\dim_K \vs{AB} \ge \kappa_{K,L}(r,s)$. As in Theorem~\ref{linear Kneser}, let $H$ be the stabilizer of the subspace $\vs{AB}$, i.e.
$$
H = \{x \in L \mid x\vs{AB} \subset \vs{AB}\}.
$$
Then of course, $H$ is a subfield of $L$ containing $K$, and we have
$$
H \vs{AB} = \vs{AB}.
$$
We shall apply Theorem~\ref{linear Kneser} to the pair $\vs{HA}$, $\vs{HB}$ of $K$-subspaces of $L$. The first observation is that this pair has the same product as the pair $A,B$:
$$
\vs{\vs{HA}\vs{HB}} = \vs{HAB} = \vs{AB}.
$$
In particular, the stabilizer of the product is still $H$. By Theorem~\ref{linear Kneser}, we obtain
$$
\dim_K \vs{AB} \ge \dim_{K}\vs{HA}+\dim_{K}\vs{HB}-\dim_{K}H.
$$
Let $g = \dim_K H$. Factoring $g$ in the above formula, we get
\begin{equation}\label{ge}
\dim_K \vs{AB} \ge (\frac{\dim_{K}\vs{HA}}{g}+\frac{\dim_{K}\vs{HB}}{g}-1)g.
\end{equation}
Now, $\vs{HA}$ is an $H$-subspace of $L$, and therefore $\dim_K \vs{HA}$ is a \textit{multiple} of $\dim_K H=g$. Moreover, the integer $(\dim_{K}\vs{HA})/{g}$ is greater than or equal to $(\dim_{K} A)/{g}={r}/{g}$. It follows that
$$
\frac{\dim_{K}\vs{HA}}{g} \ge \big\lceil \frac{r}{g} \big\rceil.
$$
The same estimate holds with $B,s$ replacing $A,r$, respectively. Plugging this information into inequality (\ref{ge}), we get
$$
\dim_K \vs{AB} \ge (\lceil {r}/{g} \rceil+\lceil {s}/{g} \rceil-1)g.
$$
Finally, given that $g$ is the dimension of an intermediate field $K \subset H \subset L$, we have
$$
(\lceil {r}/{g} \rceil+\lceil {s}/{g} \rceil-1)g \ge \kappa_{K,L}(r,s),
$$
by definition of this $\kappa$-function. It follows that $\dim_K\vs{AB} \ge \kappa_{K,L}(r,s)$. We have now shown, as claimed, that
$$
\mu_{K,L}(r,s) \ge \kappa_{K,L}(r,s)
$$
for all positive integers $r,s \le \dim_K L$.
\end{proof}

\section{Optimality}\label{optimality}
It remains to prove inequality $\mu_{K,L}(r,s) \le \kappa_{K,L}(r,s)$ of Theorem~\ref{kappaKL}. This is a construction problem. Given positive integers $r,s \le \dim_K L$, we must exhibit a pair of $K$-subspaces $A,B \subset L$ with $\dim_K A = r$, $\dim_K B = s$ and $\dim_K \vs{AB} \le \kappa_{K,L}(r,s)$. We start with a lemma on simple extensions.

\begin{lemma}\label{simple} Let $H \subset L$ be a commutative field extension, let $\alpha \in L$ and set $M=H(\alpha)$. Then, for all positive integers $r,s \le \dim_H M$, we have
$$
\mu_{H,M}(r,s) \le r+s-1.
$$
\end{lemma}
\begin{proof} Assume first that $\alpha$ is transcendental over $H$. Given integers $r,s \ge 1$, let $A = \vs{1,\alpha, \ldots, \alpha^{r-1}}$ be the $H$-subspace of $M$ spanned by the first $r$ powers of $\alpha$, and similarly let $B = \vs{1,\alpha, \ldots, \alpha^{s-1}}$. Then $\dim_H A = r$, $\dim_H B = s$ and $\dim_H \vs{AB} = \dim_H \vs{1,\alpha, \ldots, \alpha^{r+s-2}}=r+s-1$.

Assume now that $\alpha$ is algebraic over $H$, of degree $[M:H]=m$. In particular, the set $\{1,\alpha, \ldots, \alpha^{m-1}\}$ is an $H$-basis of $M$. Given positive integers $r,s \le m$, let $A = \vs{1,\alpha, \ldots, \alpha^{r-1}}$ and  $B = \vs{1,\alpha, \ldots, \alpha^{s-1}}$ as above. Then $\dim_H A = r$, $\dim_H B = s$, and $\dim_H \vs{AB} \le r+s-1$ since $\vs{AB}$ is spanned by the set $\{\alpha^i\}_{0 \le i \le r+s-2}$.

In either case, our explicit pair of subspaces $A,B$ yields the desired estimate $\mu_{H,M}(r,s) \le r+s-1$.
\end{proof}

\bigskip

As a side remark, note that the above formula remains valid if either $r=0$ or $s=0$, but not if both $r=s=0$. Using the Primitive Element Theorem for separable extensions, here is a consequence that we shall need.

\begin{proposition}\label{ssp} Let $H \subset L$ be a commutative field extension which is separable or contains a transcendental element. Then, for all positive integers $r,s \le \dim_H L$, we have
$$
\mu_{H,L}(r,s) \le r+s-1.
$$
\end{proposition}
\begin{proof} If $L$ contains a transcendental element $\alpha$, we are done by the lemma above. (Indeed, with $M=H(\alpha)$ we have $\mu_{H,L}(r,s) \le \mu_{H,M}(r,s) \le r+s-1.$) Assume now that $L$ is algebraic and separable over $H$. Given positive integers $r,s \le \dim_H L$, let $U \subset L$ be any linearly independent set of size $\max\{r,s\}$. Set $L_0 = H(U)$, the subfield of $L$ generated by $U$ over $H$. It follows from the present assumptions on $L$, that $L_0$ is a finite and separable extension of $H$, with $[L_0:H]=m \ge \max\{r,s\}$. By the Primitive Element Theorem, there exists an element $\alpha \in L_0$ such that $L_0 = H(\alpha)$. We now conclude with Lemma~\ref{simple}.
\end{proof}

\bigskip

The above result is in fact valid without any separability hypothesis, as shown in Section~\ref{removing} with a little longer argument. However, the present version is sufficient to help us conclude the proof of Theorem~\ref{kappaKL}.

\bigskip

\noindent
\begin{proof}[Proof of inequality $\mu_{K,L}(r,s) \le \kappa_{K,L}(r,s)$] Let $r,s$ be positive integers not exceeding $[L:K]$. Let $h_0=[H:K]$ be the $K$-dimension of a finite-dimensional intermediate field extension $K \subset H \subset L$ for which $\kappa_{K,L}(r,s)$ attains its value, i.e. such that
$$
\kappa_{K,L}(r,s) = (\lceil r/h_0\rceil + \lceil s/h_0\rceil -1)h_0.
$$
(Note that such an $h_0$ exists and cannot exceed $r+s-1$ since, using $h=1$ in the definition of $\kappa_{K,L}$, we have $\kappa_{K,L}(r,s) \le r+s-1$.) Set $r_0 = \lceil r/h_0\rceil$, $s_0 = \lceil s/h_0\rceil$. Of course $\lceil r/h_0\rceil, \lceil s/h_0\rceil \le [L:K]/h_0 = [L:H]$. From the hypotheses on the extension $L$ over $K$, it follows that $L$, as an extension over $H$, is either separable or else contains a transcendental element. By Proposition~\ref{ssp}, we have $\mu_{H,L}(r_0,s_0) \le r_0+s_0-1$. Thus there exist $H$-subspaces $A_0,B_0 \subset L$ such that
\begin{eqnarray*}
\dim_H A_0 & = & r_0, \\
\dim_H B_0 & = & s_0, \\
\dim_H \vs{A_0B_0} & \le & r_0+s_0-1.
\end{eqnarray*}
Now, viewed as $K$-subspaces of $L$, their dimensions are multiplied by $h_0$. Thus, we have
\begin{eqnarray*}
\dim_K A_0 & = & r_0 h_0, \\
\dim_K B_0 & = & s_0 h_0, \\
\dim_K \vs{A_0B_0} & \le & (r_0+s_0-1)h_0 \; = \; \kappa_{K,L}(r,s).
\end{eqnarray*}
Therefore $\mu_{K,L}(r_0 h_0,s_0 h_0) \le \kappa_{K,L}(r,s)$. Now $r \le r_0 h_0$, $s \le s_0 h_0$, and clearly the function $\mu_{K,L}(r,s)$ is nondecreasing in each variable. It follows that 
$$\mu_{K,L}(r,s) \le \mu_{K,L}(r_0 h_0,s_0 h_0) \le \kappa_{K,L}(r,s),
$$ 
as claimed. The proof of Theorem~\ref{kappaKL} is now complete.
\end{proof}

\section{Examples}\label{examples} We now give three examples illustrating Theorem~\ref{kappaKL}.

\begin{example}[Transcendental extensions]\label{trans} Assume that $L$ is a purely transcendental extension of $K$. In that case, the unique finite-dimensional intermediate extension $K \subset H \subset L$ is $H=K$ itself. It follows that
$\kappa_{K,L}(r,s) = r+s-1$ and thus, by Theorem~\ref{kappaKL}, we have
$$
\mu_{K,L}(r,s) = r+s-1
$$
for all positive integers $r,s$. (See also Theorem 6.3 and the remark following it in \cite{linear}.) 
\end{example}

\begin{example}[A linear version of the Cauchy-Davenport Theorem]\label{cauchy} Let $K \subset L$ be a commutative field extension of prime degree $[L:K] = p$. Asssume that char($K$) is distinct from $p$. Then, for all $1 \le r,s \le p$, we have
\begin{equation}\label{degree p}
\mu_{K,L}(r,s) = \min\{r+s-1,p\}.
\end{equation}
(Compare with the original Cauchy-Davenport Theorem in Section~\ref{link}.) Indeed, by our assumption char$(K) \not= p$, the extension is separable. Thus Theorem~\ref{kappaKL} applies and gives $\mu_{K,L}(r,s) = \kappa_{K,L}(r,s)$. Finally, since the only intermediate fields $K \subset H \subset L$ are $H=K$ and $H=L$, we have $\kappa_{K,L}(r,s) = \min\{r+s-1,p\}$ by definition of this function.
\end{example}
Actually, formula~(\ref{degree p}) also holds if char$(K)=p$, as we shall show in a future publication.

\begin{example}[An extension of degree 16]\label{degree 16} Consider the extension $\Q \subset \Q(\sqrt[16]{2})$. This is a separable extension of degree 16, obviously containing intermediate extensions of degree 2, 4 and 8. It follows that, for all $1 \le r,s \le 16$, we have
$$
\mu_{\Q,\Q(\sqrt[16]{2})}(r,s) = \kappa_{\Q,\Q(\sqrt[16]{2})}(r,s) = \min_{h|16}(\lceil r/h\rceil + \lceil s/h\rceil -1)h. 
$$
This is exactly the classical Hopf-Stiefel function $r \circ s$ \cite{hopfstiefel}. We now tabulate this function in order to sense its quite complicated behavior. The value of $r \circ s$ is the coefficient in row $\; r$ and column $\; s$ of the matrix below:
$$
\left(
\begin{array}{cccccccccccccccc}
 1 & 2 & 3 & 4 & 5 & 6 & 7 & 8 & 9 & 10 & 11 & 12 & 13 & 14 & 15 & 16 \\
 2 & 2 & 4 & 4 & 6 & 6 & 8 & 8 & 10 & 10 & 12 & 12 & 14 & 14 & 16 & 16 \\
 3 & 4 & 4 & 4 & 7 & 8 & 8 & 8 & 11 & 12 & 12 & 12 & 15 & 16 & 16 & 16 \\
 4 & 4 & 4 & 4 & 8 & 8 & 8 & 8 & 12 & 12 & 12 & 12 & 16 & 16 & 16 & 16 \\
 5 & 6 & 7 & 8 & 8 & 8 & 8 & 8 & 13 & 14 & 15 & 16 & 16 & 16 & 16 & 16 \\
 6 & 6 & 8 & 8 & 8 & 8 & 8 & 8 & 14 & 14 & 16 & 16 & 16 & 16 & 16 & 16 \\
 7 & 8 & 8 & 8 & 8 & 8 & 8 & 8 & 15 & 16 & 16 & 16 & 16 & 16 & 16 & 16 \\
 8 & 8 & 8 & 8 & 8 & 8 & 8 & 8 & 16 & 16 & 16 & 16 & 16 & 16 & 16 & 16 \\
 9 & 10 & 11 & 12 & 13 & 14 & 15 & 16 & 16 & 16 & 16 & 16 & 16 & 16 & 16 & 16 \\
 10 & 10 & 12 & 12 & 14 & 14 & 16 & 16 & 16 & 16 & 16 & 16 & 16 & 16 & 16 & 16 \\
 11 & 12 & 12 & 12 & 15 & 16 & 16 & 16 & 16 & 16 & 16 & 16 & 16 & 16 & 16 & 16 \\
 12 & 12 & 12 & 12 & 16 & 16 & 16 & 16 & 16 & 16 & 16 & 16 & 16 & 16 & 16 & 16 \\
 13 & 14 & 15 & 16 & 16 & 16 & 16 & 16 & 16 & 16 & 16 & 16 & 16 & 16 & 16 & 16 \\
 14 & 14 & 16 & 16 & 16 & 16 & 16 & 16 & 16 & 16 & 16 & 16 & 16 & 16 & 16 & 16 \\
 15 & 16 & 16 & 16 & 16 & 16 & 16 & 16 & 16 & 16 & 16 & 16 & 16 & 16 & 16 & 16 \\
 16 & 16 & 16 & 16 & 16 & 16 & 16 & 16 & 16 & 16 & 16 & 16 & 16 & 16 & 16 & 16
\end{array}
\right)
$$
For example, one finds $11 \circ 4 = 12$ and $11 \circ 5 = 15$. The fact that the lower antidiagonal part of the matrix is constant and equal to 16 is part of the following more general phenomenon. 
\end{example}

\begin{remark}
If $[L:K]=n$, then $\kappa_{K,L}(r,s) = n$ whenever $r,s \le n$ and $\;r+s \ge n+1$. 
\end{remark}
Indeed, denote $f_h(r,s) = (\lceil r/h \rceil+\lceil s/h \rceil -1)h$. Then $\kappa_{K,L}(r,s)=\min_h f_h(r,s)$, where $h$ runs over a certain set of divisors of $n$, namely the $K$-degrees of intermediate extensions. If $r+s \ge n+1$, then $f_h(r,s) \ge n+1-h$. But since $f_h(r,s)$ is a multiple of $h$, it follows that $f_h(r,s) \ge n+h-h=n$. Finally, with $h=n$ we get $f_n(r,s) = n$, and the formula follows.

\section{Finite Galois extensions}\label{galois}

In this section we consider the case of a finite Galois extension $K \subset L$ with Galois group $G$, and attempt to compare the function $\kappa_G$ from group theory to its linear version $\kappa_{K,L}$. 

By basic Galois theory, there is a bijection between intermediate extensions $K \subset H \subset L$ and subgroups of $G=Gal(L/K)$, namely $H \mapsto Gal(L/H)$. The cardinality of the subgroup of $G$ corresponding to $H$ is given by the formula
$$
|Gal(L/H)|=[L:H]=[L:K]/[H:K].
$$
Recall that $\kappa_G(r,s)$ is defined, in the case at hand, by minimizing the expression $$(\lceil r/h\rceil + \lceil s/h\rceil -1)h$$ over all subgroup cardinalities $h=|Gal(L/H)|=[L:H]$. However, in the definition of $\kappa_{K,L}(r,s)$, the minimum is rather taken over the numbers $h = [H:K]$. Thus, the functions $\kappa_{K,L}(r,s)$ and $\kappa_G(r,s)$ cannot be directly compared in general, except in the particular case where all divisors of $|G|$ happen to be subgroup cardinalities; this occurs for instance if $G$ is abelian or a $p$-group. This observation yields the following consequences of Theorem~\ref{kappaKL}.

\begin{corollary}\label{Galois} Let $K \subset L$ be a Galois extension with finite Galois group $G$ of order $n$. Assume that every divisor $d$ of $n$ is a subgroup cardinality. Then, for all positive integers $r,s \le n=[L:K]$, we have
$$
\mu_{K,L}(r,s) = \kappa_G(r,s) = \min_{d|n}(\lceil r/d \rceil + \lceil s/d \rceil -1 )d.
$$
\end{corollary}
Assuming further that $G$ is abelian, and using Theorem~\ref{abelian}, we get an equality on the level of $\mu$-functions.

\begin{corollary}\label{abelian Galois} Let $K \subset L$ be a Galois extension with finite abelian Galois group $G$ of order $n$. Then, for all positive integers $r,s \le n=[L:K]$, we have
\begin{equation}\label{mu}
\mu_{K,L}(r,s) = \mu_G(r,s).
\end{equation}
\end{corollary}
However, note that equality~(\ref{mu}) does not hold in general if $G$ is nonabelian, even if all divisors of $|G|$ are subgroup cardinalities. For instance, for the nonabelian group $G = \Z/7\Z \rtimes \Z/3Z$ of order 21, it is known that $\mu_G(5,9)=\kappa_G(5,9)+1 = 13$; this provides, by Corollary~\ref{Galois}, a counterexample to equality~(\ref{mu}).

\section{The small products property}\label{removing}

In this section we show that Proposition~\ref{ssp} is valid in an arbitrary commutative field extension $H \subset L$, not necessarily separable. Indeed, we shall prove that, for all positive integers $r,s \le [L:H]$, there exist $H$-subspaces $A,B$ of $L$ with $\dim_H A =r$, $\dim_H B =s$ and $\dim_H \vs{AB} \le r+s-1$. We might call this the \textit{small products property}, in analogy with the small sumsets property for groups.

\begin{proposition}\label{ssp general} Let $H \subset L$ be a commutative field extension. Then, for all positive integers $r,s \le \dim_H L$, we have
$$
\mu_{H,L}(r,s) \le r+s-1.
$$
\end{proposition}
\begin{proof} As in the proof of Proposition~\ref{ssp}, we are done if $L$ contains a transcendental element over $H$. Assume now that $L$ is algebraic over $H$. If $[L:H]$ is infinite, then $L$ contains intermediary extensions $H \subset L' \subset L$ with $[L':H]$ finite but arbitrarily large. (Indeed, take $L'=H(u_1,\ldots,u_n)$ for any finite choice of $u_1,\ldots,u_n \in L$.) Hence we may further assume that $[L:H]$ is finite. Let $H \subset M \subset L$ be an intermediate extension for which the statement of the proposition is true, namely satisfying
\begin{equation}\label{eq: ssp}
\mu_{H,M}(r_0,s_0) \le r_0+s_0-1
\end{equation}
for all $1 \le r_0,s_0 \le [M:H]$. Such extensions exist, for instance $M=H$. We may further assume that $M$ is maximal for this small products property. If $M=L$ we are done. If not, let $\alpha \in L \setminus M,$ say of degree $d$ over $M$. For the record, the set $\{1,\alpha,\ldots,\alpha^{d-1}\}$ is an $M$-basis of $M(\alpha)$. We shall show that the statement of the proposition still holds for the extension $H \subset M(\alpha)$, a contradiction to the maximality of $M$.

Let $r,s \le [M(\alpha):H] = [M:H]d$. Performing a slightly modified euclidean division by $[M:H]$, we may write
\begin{eqnarray*}
r & = & q_1 [M:H] + r_0,\\
s & = & q_2 [M:H] + s_0,
\end{eqnarray*}
with remainders $1 \le r_0, s_0 \le [M:H]$ and quotients $q_1,q_2 \le d-1$. 

Since $\mu_{H,M}(r_0,s_0) \le r_0+s_0-1$, we may choose $H$-subspaces $A_0,B_0 \subset M$ such that
$$
\begin{array}{lll}
\dim_H A_0 & = & r_0,\\
\dim_H B_0 & = & s_0,\\
\dim_H \vs{A_0B_0} & \le & r_0+s_0-1.
\end{array}
$$
We may assume $q_1+q_2 \ge 1$, for otherwise $r=r_0,s=s_0$ and we are done in this case by assumption on $M$. We now define
\begin{eqnarray*}
A & = & M\cdot\{1,\alpha, \ldots, \alpha^{q_1-1}\} \oplus A_0\cdot \alpha^{q_1},\\
B & = & M\cdot\{1,\alpha, \ldots, \alpha^{q_2-1}\} \oplus B_0\cdot \alpha^{q_2},
\end{eqnarray*}
provided $q_1, q_2 \ge 1$. If $q_1=0$ or $q_2=0$, we simply set $A=A_0$ or $B=B_0$, respectively. 
In all cases, viewing $A,B$ as vector spaces over $H$, we have 
\begin{eqnarray*}
\dim_H A = q_1[M:H]+r_0 = r,\\
\dim_H B = q_2[M:H]+s_0 = s.
\end{eqnarray*}
(Recall that $1,\alpha,\ldots,\alpha^{d-1}$ are linearly independent over $M$, that $q_1,q_2 \le d-1$ and that $A_0,B_0 \subset M$.) Now, taking the product of $A$ and $B$, it is plain that we get
$$
\vs{AB} \subset M\cdot\{1,\alpha, \ldots, \alpha^{q_1+q_2-1}\} \oplus \vs{A_0B_0}\cdot \alpha^{q_1+q_2}.
$$
It follows that
$$
\dim_H\vs{AB} \le (q_1+q_2)[M:H]+(r_0+s_0-1) = r+s-1,
$$
and the proof of the proposition is complete.
\end{proof}

\section{Two conjectures}\label{conjectures}

In Theorems~\ref{kappaKL} and \ref{linear Kneser}, the extension $K \subset L$ is assumed to have all its algebraic elements separable. Are these results still valid without this hypothesis? The answer for Theorem~\ref{linear Kneser} is conjectured in \cite{hou} to be positive.
\begin{conjecture}(X.D.Hou)\label{conj: hou} Let $K\subset L$ be a commutative field
extension, and let $A,B\subset L$ be nonzero finite-dimensional $K$-subspaces of $L$. Let $H$ be the stabilizer of $\langle AB\rangle.$ 
Then 
\begin{equation*}
\dim _{K}\langle AB\rangle \geq \dim _{K}A+\dim _{K}B-\dim _{K}H.
\end{equation*}
\end{conjecture}
It is shown in \cite{hou} that the statement of the conjecture holds for $\dim_K A \le 5$.

\bigskip

It remains to decide whether the separability hypothesis in Theorem~\ref{kappaKL} can be removed. We conjecture that this is the case.
\begin{conjecture}\label{conj: kappaKL}
Let $K \subset L$ be a commutative field extension. Then, for all positive integers $r,s \le \dim_K L$, one should have
$$
\mu_{K,L}(r,s) = \kappa_{K,L}(r,s).
$$
\end{conjecture}

This conjecture in fact follows from Conjecture~\ref{conj: hou}. Indeed, our proof of Theorem~\ref{kappaKL} relies on both Theorem~\ref{linear Kneser} and Proposition~\ref{ssp}. Removing the separability hypotheses in these two results yields Conjecture~\ref{conj: hou} and Proposition~\ref{ssp general}, respectively. With the latter statements, our proof of Theorem~\ref{kappaKL} becomes a derivation of Conjecture~\ref{conj: kappaKL} from Conjecture~\ref{conj: hou}. In particular, by the above-mentioned result in \cite{hou}, Conjecture~\ref{conj: kappaKL} holds at least for $r \le 5$.

Of course, by Theorem~\ref{kappaKL}, Conjecture~\ref{conj: kappaKL} holds for all separable extensions, and in particular in characteristic 0.

\vspace{4mm}
{\footnotesize {\bf Acknowledgment:} During the preparation of this paper, the first
author has partially benefited from a research contract with the Fonds National
Suisse pour la Recherche Scientifique.}

\vspace{6mm}

{ \centerline{\large Authors' Addresses}

\vspace{4mm}

\noindent
{\sc Shalom Eliahou, C\'edric Lecouvey},\\
LMPA Joseph Liouville, FR CNRS 2956\\
Universit\'{e} du Littoral C\^{o}te d'Opale\\
50 rue F. Buisson, B.P. 699\\
F-62228 Calais cedex, France\\
e-mail: eliahou@lmpa.univ-littoral.fr, lecouvey@lmpa.univ-littoral.fr

\vspace{3mm}

\noindent
{\sc Michel Kervaire},\\
D\'epartement de Math\'ematiques\\ 
Universit\'e de Gen\`{e}ve\\
2-4, rue du Li\`{e}vre, Case postale 64\\ 
CH-1211 Gen\`{e}ve 24, Suisse\\
e-mail: Michel.Kervaire@math.unige.ch

}

\end{document}